\documentclass[11pt]{article}
\usepackage{amsfonts}
\usepackage{amsmath,color}
\usepackage[makeroom]{cancel}

\newcommand{\beq}{\begin{equation}}
\newcommand{\eeq}{\end{equation}}
\newcommand{\bea}{\begin{eqnarray}}
\newcommand{\eea}{\end{eqnarray}}
\newcommand{\beas}{\begin{eqnarray*}}
\newcommand{\eeas}{\end{eqnarray*}}

\newtheorem{theorem}{Theorem}[section]

\newtheorem{proposition}[theorem]{Proposition}

\newtheorem{lemma}[theorem]{Lemma}
\newtheorem{remark}[theorem]{Remark}
\newtheorem{example}[theorem]{Example}
\newtheorem{examples}[theorem]{Examples}
\newtheorem{foo}[theorem]{Remarks}

\newenvironment{proof}{\addvspace{\medskipamount}\par\noindent{\it
Proof}.}
{\unskip\nobreak\hfill$\Box$\par\addvspace{\medskipamount}}

\newcommand{\ee}{\ell}

\newcommand{\bM}{\mathbb M}

\newcommand{\Ho}{\mathcal H}
\newcommand{\V}{\mathcal V}

\newcommand{\M}{\mathbb M}

\newcommand{\ve}{\varepsilon}

\newcommand{\Ric}{\mathfrak {Ric}}
\newcommand{\ch}{\mathcal H}
\newcommand{\J}{\mathfrak J}

\newcommand{\ep}{\varepsilon}
\title{Transverse Weitzenb\"ock formulas and curvature dimension inequalities on Riemannian foliations with totally geodesic leaves}
\author{Fabrice Baudoin, Bumsik Kim, Jing Wang}

\date{Department of Mathematics, Purdue University \\
 West Lafayette, IN, USA}

\begin{document}
\maketitle

\begin{abstract}
We prove a family of  Weitzenb\"ock  formulas on a Riemannian foliation with totally geodesic leaves. These Weitzenb\"ock formulas are naturally parametrized by the canonical variation of the metric. As a consequence, under natural geometric conditions,  the horizontal Laplacian  satisfies a generalized curvature dimension inequality.  Among other things, this curvature dimension inequality implies  Li-Yau estimates  for positive solutions of the horizontal heat equation, sharp eigenvalue estimates and a sub-Riemannian Bonnet-Myers compactness theorem whose assumptions only rely on the intrinsic geometry of the horizontal distribution.
\end{abstract}

\

\textbf{Keywords:} Sub-Riemannian geometry, Weitzenb\"ock formula, Bochner method, Riemannian foliation

\

\textbf{MSC2010: 53C12,53C17}

\tableofcontents


\section{Introduction}

In the recent few years, there has been a great deal of interest in developing geometric analysis methods in sub-Riemannian geometry that parallel the methods which  are available in Riemannian geometry (see for instance \cite{agrachev,agrachev2} and the references therein). One of these fundamental methods is the Bochner method which connects the geometry and topology of the ambient manifold to the analysis of a second order differential operator, the Laplace-Beltrami operator. In all generality, there is no canonical second-order operator on a sub-Riemannian manifold, see \cite{GL} for a closely related discussion. However, in many interesting situations, like in the Hopf fibrations, the sub-Riemannian geometry of interest is associated to the horizontal distribution of a Riemannian foliation. In that particular case, there is a canonical sub-Riemannian Laplacian: the horizontal Laplacian of the foliation. With such an operator in hand, it is then natural to wonder if a suitable analogue of  Bochner's formalism would make the bridge between the analysis of this horizontal Laplacian and the sub-Riemannian geometry and the topology of the ambient manifold. 

\

In the present paper, we prove that on a Yang-Mills type Riemannian foliation with totally geodesic leaves and a bracket generating horizontal distribution many fundamental tools are available to study the horizontal Laplacian and the associated  sub-Riemannian geometry. These tools include Li-Yau type gradient estimates for positive solutions of the heat equation and associated Harnack inequalities, volume comparison theorems for the metric balls, sharp Sobolev inequalities, sharp first eigenvalue estimates and corresponding rigidity results. Our work here builds on several previous works, notably \cite{BBG,BBGM,BG}, where it is proved that, if a certain curvature dimension inequality and some further assumptions are satisfied, then the above tools are automatically available. So, our main contribution here is to prove that this inequality and assumptions are satisfied in any Riemannian foliation with totally geodesic leaves if the following natural geometric conditions are fulfilled: completeness of the metric, uniform bracket generating condition, global lower bound on the horizontal Ricci curvature and  global upper bound on the torsion of the Bott connection. Proving the curvature dimension estimate under these assumptions is not easy and our analysis relies on new taylor-made Bochner-Weitzenb\"ock type inequalities for a suitable one-parameter family of sub-Laplacians on one-forms.

\

It is now time to give more details on our contribution. In the work \cite{BG} the authors proved that on a sub-Riemannian manifold with transverse symmetries, assuming natural geometric conditions,  the sub-Laplacian satisfies a generalized curvature dimension inequality. Among other things, this curvature  dimension estimate implies Li-Yau inequalities for positive solutions of the heat equation \cite{BBG,BG}, Gaussian lower and upper bounds for the subelliptic heat kernel \cite{BBG,BG}, log-Sobolev and isoperimetric inequalities \cite{BB, BK}, volume and distance comparison estimates \cite{BBGM} and a Bonnet-Myers type theorem \cite{BG}.  Recently, it has been pointed out by Elworthy \cite{Elworthy} that sub-Riemannian manifolds with transverse symmetries can be seen as Riemannian manifolds with bundle like metrics which are  foliated by totally geodesic leaves. The goal of the present work is two-fold:
\begin{itemize}
\item We actually prove that on any Yang-Mills type Riemannian foliation with bundle like metric and  totally geodesic leaves, under natural geometric conditions, the horizontal Laplacian satisfies the curvature dimension estimate introduced in \cite{BG}. As a consequence, all the results proved in \cite{BB,BBG,BBGM, BG, BG2, BK, Kim} apply in this more general case.
\item We  simplify the original approach of \cite{BG} by working out new Weitzenb\"ock type identities for the horizontal Laplacian which we think are interesting in themselves. These Weitzenb\"ock  identities easily imply not only the curvature dimension estimate but also the stochastic completeness of the heat semigroup, which is a crucial ingredient to run the  machinery  developed in \cite{BG}.
\end{itemize}

\

The paper is organized as follows. In Section 2, we give the basic definitions and conventions that will be used throughout the text. In Section 3, we introduce a canonical one parameter family of horizontal Laplacians on one-forms and prove Weitzenb\"ock-Bochner's type inequalities for this family of horizontal Laplacians. In Section 4, we prove the generalized curvature dimension inequality. We point out that, unlike many previous works on Riemannian foliations (see \cite{Tondeur, Tondeur2} and the references therein), our results in Section 4 actually concern the sub-Riemannian geometry associated to the horizontal distribution and are not restricted to the transversal  geometry of the foliation.

\

\textbf{Acknowledgments:} F. Baudoin would like to thank Pr. K. D. Elworthy for stimulating discussions. All the authors thank an anonymous referee for an exceptional level of reading and valuable comments. The paper in this form greatly benefited from those comments.

\section{Preliminaries}

Let $\M$ be a smooth, connected  manifold with dimension $n+m$. We assume that $\bM$ is equipped with a Riemannian foliation $\mathcal{F}$ with bundle-like metric $g$ and totally geodesic  $m$-dimensional leaves (see the classical monograph by Tondeur \cite{Tondeur} for the basic properties of such foliations).  Prototype examples of such foliations arise in the context of Riemannian submersions. If $\pi :\M \to \mathbb{N}$ is a Riemannian submersion whose fibers are totally geodesic in $\M$, then the foliation of $\M$ by the fibers of $\pi$ gives a totally geodesic foliation with bundle-like metric on $\M$. Actually, any totally geodesic foliation with bundle like metric can locally be described by such submersion (see \cite{IHP}).

\

As we will see,  the totally geodesic foliation framework is the good  setting to extend the generalized curvature dimension developed in \cite{BG}. An interesting class of examples of homogeneous spaces that fall in the present framework and not in the one of \cite{BG} is the following. This class of examples encompasses all the generalized Hopf fibrations (see Chapter 9, Section H in \cite{Besse}).

\begin{example}
Let $\mathbf{G}$ be a Lie group, and $\mathbf{H},\mathbf{K}$ be two compact subgroups of $\mathbf{G}$ with $\mathbf{K} \subset \mathbf{H}$. Then, we have a natural fibration given by the coset map
\begin{align*}
\begin{array}{llll}
\pi: & \mathbf{G} / \mathbf{K} & \to & \mathbf{G} / \mathbf{H}  \\
 & \alpha \mathbf{K} & \to & \alpha \mathbf{H} ,
\end{array}
\end{align*}
where the fiber is $\mathbf{H} / \mathbf{K}$. From \cite{Be}, there exist $\mathbf{G}$-invariant metrics on respectively  $\mathbf{G} / \mathbf{K} $ and  $ \mathbf{G} / \mathbf{H} $ that make $\pi$ a Riemannian submersion with totally geodesic fibers isometric to $\mathbf{H} / \mathbf{K}$. It is then easily seen that  $\mathbf{G} / \mathbf{K}$ is endowed with a structure of sub-Riemannian manifold with transverse symmetries in the sense of \cite{BG} if and only if $\mathbf{K}$ is the trivial subgroup $\{ e \}$. In that case, the group of transverse symmetries is isometric to $\mathbf{H}$.
\end{example}

The sub-bundle $\mathcal{V}$ formed by vectors tangent to the leaves is referred  to as the set of \emph{vertical directions}. The sub-bundle $\mathcal{H}$ which is normal to $\mathcal{V}$ is referred to as the set of \emph{horizontal directions}.   The metric $g$ can be split as
\[
g=g_\mathcal{H} \oplus g_{\mathcal{V}},
\]
and for later use, we introduce the one-parameter family of Riemannian metrics:
\[
g_{\varepsilon}=g_\mathcal{H} \oplus  \frac{1}{\varepsilon }g_{\mathcal{V}}, \quad \varepsilon >0,
\]
 which is going to play a pervasive role in the sequel. It is called the canonical variation of $g$, see Chapter 9 G in the monograph by Besse \cite{Besse}.

\

There is a canonical connection on $\M$, the Bott connection, which is given as follows:

\[
\nabla_X Y =
\begin{cases}
\pi_{\mathcal{H}} ( \nabla_X^R Y) , X,Y \in \Gamma^\infty(\mathcal{H}) \\
\pi_{\mathcal{H}} ( [X,Y]),  X \in \Gamma^\infty(\mathcal{V}), Y \in \Gamma^\infty(\mathcal{H}) \\
\pi_{\mathcal{V}} ( [X,Y]),  X \in \Gamma^\infty(\mathcal{H}), Y \in \Gamma^\infty(\mathcal{V}) \\
\pi_{\mathcal{V}} ( \nabla_X^R Y) , X,Y \in \Gamma^\infty(\mathcal{V})
\end{cases}
\]
where $\nabla^R$ is the Levi-Civita connection and $\pi_\mathcal{H}$ (resp. $\pi_\mathcal{V}$) the projection on $\mathcal{H}$ (resp. $\mathcal{V}$). It is easy to check that for every $\varepsilon >0$, this connection satisfies $\nabla g_\varepsilon=0$.

\

For local computations, it will be convenient to work in local frames that are adapted to the geometry of the foliation.

\begin{lemma}\label{frame}
Let $x \in \M$. Around $x$, there exist a local orthonormal horizontal frame $\{X_1,\cdots,X_n \}$ and a local orthonormal vertical frame $\{Z_1,\cdots,Z_m \}$ such that the following structure relations hold
\[
[X_i,X_j]=\sum_{k=1}^n \omega_{ij}^k X_k +\sum_{k=1}^m \gamma_{ij}^k Z_k
\]
\[
[X_i,Z_k]=\sum_{j=1}^m \beta_{ik}^j Z_j,
\]
where $\omega_{ij}^k,  \gamma_{ij}^k,  \beta_{ik}^j $ are smooth functions such that:
\[
 \beta_{ik}^j=- \beta_{ij}^k.
\]
Moreover, at $x$, we have
\[
 \omega_{ij}^k=0,  \beta_{ij}^k=0.
\]
\end{lemma}

\begin{proof}
Since the statement is local, we can assume that the Riemannian foliation comes from  a Riemannian submersion with totally geodesic fibers. We fix  $x \in \M$ throughout the proof.

Let $X_1,\dots, X_n$ be a local orthonormal horizontal frame around $x$ consisting of basic vector fields for the submersion. We can assume that, at $x$, $\nabla_{X_i} X_j =0$. Let now $Z_1,\dots,Z_m$ be any local orthonormal vertical frame around $x$. Since $X_i$ is basic, the vector field $[X_i,Z_m]$ is tangent to the leaves. We  write the structure constants in that local frame:
\[
[X_i,X_j]=\sum_{k=1}^n \omega_{ij}^k X_k +\sum_{k=1}^m \gamma_{ij}^k Z_k
\]
\[
[X_i,Z_k]=\sum_{j=1}^m \beta_{ik}^j Z_j,
\]
and observe that at the center $x$ of the frame, we have $\omega_{ij}^k =0$. Moreover, since $X_i$ is basic and the submersion has totally geodesic fibers, the flow generated by $X_i$ induces an isometry between the leaves (see Besse \cite{Besse}, Chapter 9), as a consequence we have the skew-symmetry,
\[
\beta_{ik}^j =-\beta_{ij}^k.
\]
Also for later use, we record the fact that in this frame the Christofell symbols of the Bott connection are given by
\begin{align*}
\begin{cases}
\nabla_{X_i} X_j =\frac{1}{2} \sum_{k=1}^n \left( \omega_{ij}^k +\omega_{ki}^j+\omega_{kj}^i\right)X_k \\
\nabla_{Z_j} X_i =0 \\
\nabla_{X_i} Z_j=\sum_{k=1}^m \beta_{ij}^{k} Z_k
\end{cases}
\end{align*}
It remains to prove that we can also assume that, at the center $x$, $\beta_{ij}^k=0$.  It is enough to prove that there exists a local orthonormal vertical frame  $Z_1,\cdots,Z_m$ such that for any horizontal field $X$, we have at $x$, $\nabla_X Z_i =0$. To this end, we use an argument inspired by  \cite{Hladky}, Corollary 2.22.

Let $u_1,\cdots,u_n$ be an orthonormal frame of $\mathcal{H}_x$ and $v_1,\cdots,v_m$ be an orthonormal frame of $\mathcal{V}_x$. Let us denote by $x_1,\cdots,x_n, z_1, \cdots , z_m$ the coordinates near $x$ obtained from the exponential map $\exp_x (\sum_{i=1}^n x_i u_i +\sum_{i=1}^m z_i v_i)$ of the Bott connection $\nabla$. We have at $x$,
\[
\nabla_{\partial_{x_i} } \partial_{z_j} +\nabla_{ \partial_{z_j} } \partial_{x_i} =0.
\]
Since at $x$, the torsion $T(\partial_{x_i} , \partial_{z_j})$ is zero because $u_i$ is horizontal and $v_j$ vertical, we deduce that at $x$
\[
\nabla_{\partial_{x_i} } \partial_{z_j} =\nabla_{ \partial_{z_j} } \partial_{x_i} =0.
\]
We now define $Z_k$ to be the orthogonal projection of $\partial_{z_k}$ onto the vertical bundle. The $Z_k$'s form a local vertical frame around $x$. If $X$ is any smooth vector field around $x$, we know that the verticality of $Z_k$ implies on one hand that $\nabla_X Z_k$ is vertical. On the other hand, we can write
\[
\nabla_X Z_k=\nabla_X (Z_k-\partial_{z_k})+ \nabla_X \partial_{z_k}
\]
The vector $\nabla_X (Z_k-\partial_{z_k})$ is horizontal and, at $x$, $ \nabla_X \partial_{z_k}=0$, thus at $x$,  $\nabla_X Z_k$ is also horizontal. This implies $\nabla_X Z_k =0$, at $x$. The Gram-Schmidt orthonormalization of $Z_1,\cdots,Z_m$ gives then the expected local vertical frame.
\end{proof}

We define the horizontal gradient $\nabla_\mathcal{H} f$ of a function $f$ as the projection of the Riemannian gradient of $f$ on the horizontal bundle. Similarly, we define the vertical gradient $\nabla_\mathcal{V} f$ of a function $f$ as the projection of the Riemannian gradient of $f$ on the vertical bundle.
The horizontal Laplacian $L$ is the generator of the symmetric Dirichlet form:
\[
\mathcal{E}_{\mathcal{H}} (f,g) =-\int_\bM \langle \nabla_\mathcal{H} f , \nabla_\mathcal{H} g \rangle_{\mathcal{H}} d\mu,
\]
where $\mu$ is the Riemannian volume. It is a diffusion operator $L$ on $\bM$ which is symmetric  on $C^\infty_0 (\bM)$ with respect  to the volume measure $\mu$.  A routine computation shows that in the above local frame,
\[
L=\sum_{i=1}^n \nabla_{X_i}\nabla_{X_i} -\nabla_{\nabla_{X_i} X_i}.
\]
It should be noted that if we assume the horizontal distribution to be bracket-generating (which we do not until Section 4), then from H\"ormander's theorem $L$ is a hypoelliptic diffusion operator and $(\M, \Ho, g_\Ho)$ is a sub-Riemannian structure. We mention that in all generality,  even in the two-step generating case, the Popp measure of this sub-Riemannian structure  which appears as a natural volume form  on the manifold (see \cite{BR}), does not need to coincide with a constant multiple of the Riemannian volume measure. This is however the case in some interesting examples like Carnot groups of step 2 or CR Sasakian manifolds. We also mention that while the horizontal gradient is intrinsically associated to the sub-Riemannian structure $(\M, \Ho, g_\Ho)$, the vertical gradient is not: It comes from the additional choice of the vertical complement $\V$ to the horizontal distribution.

\

We now introduce some tensors that will play an important role in the sequel.

\

For $Z \in \Gamma^\infty(\V)$, there is a  unique skew-symmetric endomorphism  $J_Z:\mathcal{H}_x \to \mathcal{H}_x$ such that for all horizontal vector fields $X$ and $Y$,
\begin{align}\label{Jmap}
g_\mathcal{H} (J_Z (X),Y)= g_\mathcal{V} (Z,T(X,Y)).
\end{align}
where $T$ is the torsion tensor of $\nabla$. We then extend $J_{Z}$ to be 0 on  $\mathcal{V}_x$. Also, if $Z\in \Gamma^\infty (\ch(\M))$, from \eqref{Jmap} we set $J_Z=0$. If $Z_1,\dots,Z_m$ is a local vertical frame, the operator $\sum_{\ee=1}^m J_{Z_\ee}J_{Z_\ee}$ does not depend on the choice of the frame and shall concisely be denoted by $\mathbf{J}^2$. For instance, if $\M$ is a K-contact manifold equipped with the Reeb foliation, then $\mathbf{J}$ is an almost complex structure, $\mathbf{J}^2=-\mathbf{Id}_{\mathcal{H}}$.

\

 The horizontal divergence of the torsion $T$ is the $(1,1)$ tensor  which is defined in a local horizontal frame $X_1,\dots,X_n$ by
\[
\delta_\mathcal{H} T (X)= -\sum_{j=1}^n(\nabla_{X_j} T) (X_j,X).
\]
The $g$-adjoint of $\delta_\mathcal{H}T$ will be denoted $\delta^*_\mathcal{H} T$.
\

 By declaring a one-form to be horizontal (resp. vertical) if it vanishes on the vertical bundle $\mathcal{V}$ (resp. on the horizontal bundle $\mathcal{H}$), the splitting of the tangent space
 \[
 T_x \bM= \mathcal{H}_x \oplus \mathcal{V}_x
 \]
 gives a splitting of the cotangent space.


 The metric $g_\varepsilon$ induces  then a metric on the cotangent bundle which we still denote $g_\varepsilon$. By using similar notations and conventions as before we define pointwisely for every $\eta$ in $T^*_x \M$,
\[
\| \eta \|^2_{\varepsilon} =\| \eta \|_\mathcal{H}^2+\varepsilon \| \eta \|_\mathcal{V}^2.
\]

\

By using the duality given by the metric $g$, $(1,1)$ tensors can also be seen as linear maps on the cotangent bundle $T^* \M$. More precisely, if $A: \Gamma^\infty(T\M)\to \Gamma^\infty(T\M)$ is a $(1,1)$ tensor, we will still denote by $A$ the fiberwise linear map on the cotangent bundle which is defined as the $g$-adjoint of the dual map of $A$. Namely $A:\Gamma^\infty(T^*\M)\to \Gamma^\infty(T^*\M)$ is such that for any $\eta, \xi\in \Gamma(T^*\M)$, $\langle A\eta,\xi\rangle=\xi(A\eta^\sharp)$ where $\sharp$ is the standard musical isomorphism. The same convention will be made for any $(r,s)$ tensor. As a convention, unless explicitly mentioned otherwise in the text, the inner product duality will always be understood with respect to the reference metric $g$ (and not $g_\varepsilon$).

\

We define then the horizontal Ricci curvature $\mathfrak{Ric}_{\mathcal{H}}$ as the fiberwise symmetric linear map on one-forms such that for every smooth functions $f,g$,
\[
\langle  \mathfrak{Ric}_{\mathcal{H}} (df), dg \rangle=\mathbf{Ricci} (\nabla_\mathcal{H} f ,\nabla_\mathcal{H} g),
\]
where $\mathbf{Ricci}$ is the Ricci curvature of the connection $\nabla$.

\

If $V$ is a horizontal vector field and $\varepsilon >0$, we consider the fiberwise linear map from the space of one-forms into itself which is given for $\eta \in \Gamma^\infty(T^* \M)$ and $Y \in  \Gamma^\infty(T \M)$ by
\[
\mathfrak{T}^\varepsilon_V \eta (Y) =
\begin{cases}
\frac{1}{\varepsilon} \eta (J_Y V), \quad Y \in \Gamma^\infty(\mathcal{V}) \\
-\eta (T(V,Y)), Y  \in \Gamma^\infty(\mathcal{H})
\end{cases}
\]
and $\mathfrak{T}^\varepsilon_V=0$ when $V$ is a vertical vector field.
We observe that $\mathfrak{T}^\varepsilon_V$ is skew-symmetric for the metric $g_\varepsilon$ so that $\nabla -\mathfrak{T}^\varepsilon$ is a $g_\varepsilon$-metric connection.

\
If $Z_1,\dots,Z_m$ is a local vertical frame of the leaves as above, we denote
\[
\J(\eta)=-\sum_{\ee=1}^mJ_{Z_\ee}(\iota_{Z_\ee}d\eta_\V),
\]
 where $\eta_\V$ is the the projection of $\eta$ to the vertical cotangent bundle and $\iota$ the usual  interior product. Of course, $\J$ does not depend on the choice of the frame.

\

If $\eta$ is a one-form, we define the horizontal gradient in a local adapted frame of $\eta$ as the $(0,2)$ tensor
\[
\nabla_\mathcal{H} \eta =\sum_{i=1}^n \nabla_{X_i} \eta \otimes \theta_i.
\]
where $\theta_i, i=1,\dots, n$ is the dual of $X_i$.
We denote by $\nabla_\ch^\# \eta$ the symmetrization of $\nabla_\mathcal{H} \eta $.

\

Similarly, we will use the notation
\[
\mathfrak{T}^\varepsilon_\mathcal{H} \eta =\sum_{i=1}^n \mathfrak{T}^\varepsilon_{X_i} \eta  \otimes \theta_i.
\]

Finally. we will still denote by $L$ the covariant extension on one-forms of the horizontal Laplacian 
\[
L=-\nabla_\ch^*\nabla_\ch.
\]
\section{Bochner-Weitzenb\"ock formulas for the horizontal Laplacian}

For $\varepsilon >0$, we consider the following operator which is defined on one-forms by
\begin{align}\label{def}
\square_\varepsilon=-(\nabla_\mathcal{H} -\mathfrak{T}_\mathcal{H}^\varepsilon)^* (\nabla_\mathcal{H} -\mathfrak{T}_\mathcal{H}^\varepsilon)-\frac{1}{ \varepsilon}\mathbf{J}^2+\frac{1}{\varepsilon} \delta_\mathcal{H} T- \mathfrak{Ric}_{\mathcal{H}},
\end{align}
where the adjoint is understood with respect to the $(L^2,g_\varepsilon)$ product on sections, i.e. $\int_\M \langle\cdot,\cdot\rangle_{\varepsilon}d\mu$. It is easily seen that, in the local  frame,
\begin{align}\label{eq-L-form}
-(\nabla_\mathcal{H} -\mathfrak{T}_\mathcal{H}^\varepsilon)^* (\nabla_\mathcal{H} -\mathfrak{T}_\mathcal{H}^\varepsilon)
=\sum_{i=1}^n (\nabla_{X_i} -\mathfrak{T}^\varepsilon_{X_i})^2 - ( \nabla_{\nabla_{X_i} X_i}-  \mathfrak{T}^\varepsilon_{\nabla_{X_i} X_i}),
\end{align}

Before we  proceed to study the main geometric properties of $\square_\varepsilon$, it may be useful to explain where  this operator comes from. One of the main motivations of the present work was to explain, in the correct geometric setting, the Bochner's identities of \cite{BG}. More precisely,  its is well known that  the celebrated Bochner's formula on Riemannian manifolds is the consequence of the Weitzenb\"ock formula on one-forms, and we wanted to check if this is also the case in the framework of \cite{BG} and in the more general framework of foliations. This motivation leads to study second order  operators $\mathfrak{G}$ on one-forms such that for $f \in C_0^\infty(\M)$,
\[
dLf=\mathfrak{G} d f,
\]
where $d$ is the usual exterior derivative. Obviously there are infinitely many operators $\mathfrak{G}$ that satisfy the above intertwining, and as we will see $\square_\varepsilon$ is only one of them. It is however remarkable among such $\mathfrak{G}$'s in at least two ways:

\begin{itemize}
\item It is first remarkable that it can be written in the form $\square_\varepsilon=-(\nabla^\varepsilon)^*_\Ho \nabla_\Ho^\varepsilon +\mathcal{R}$, where $\nabla^\varepsilon$ is a connection which is metric for $g_\varepsilon$ and where $\mathcal{R}$ is $0$-th order term. This special form allows to write a Bochner's inequality for the metric $g_\varepsilon$ and use the stochastic  method to prove stochastic completeness of the semigroup generated by $L$: We refer to  section 4 for a further explanation of this fact.
\item Sending the parameter $\varepsilon$ to $\infty$, we get an operator denoted $\square_\infty$. If the Riemannian foliation comes from a Riemannian submersion, then  $\square_\infty$ can be interpreted as a lift of the Hodge-de Rham Laplacian of the base space of the submersion. Then, similarly to Proposition 3.2 in \cite{B}, it is easy to check that for any fiberwise linear map $\Lambda$ from the space of two-forms into the space of one-forms, and any $x \in \M$, we have
\begin{align*}
 & \inf_{\eta, \| \eta (x) \|_{\varepsilon}=1} \left(  \frac{1}{2} (L \| \eta \|_\varepsilon^2)(x) -\langle (\square_\infty +\Lambda \circ d )\eta (x) , \eta (x)\rangle_\varepsilon \right)  \\
 \le & \inf_{\eta, \| \eta (x) \|_{\varepsilon}=1} \left(  \frac{1}{2} (L \| \eta \|_\varepsilon^2)(x) -\left\langle \left(\square_\infty -\frac{1}{\varepsilon} \mathcal{T} \circ d\right) \eta (x) , \eta (x) \right\rangle_\varepsilon \right),
\end{align*}
where in the above notation, the torsion tensor $T$ is interpreted, by duality, as a fiberwise linear map from the space of two-forms into the space of one-forms. In Lemma \ref{lemma-sq-ep}, we will prove that $\square_\infty -\frac{1}{\varepsilon}\mathcal{T} \circ d= \square_{\varepsilon /2}$. As a consequence $\square_{\varepsilon /2}$ can be identified with the operator $\mathfrak{G}$  which satisfies  $dL=\mathfrak{G} d$, has the same principal symbol as $-\nabla^*_\Ho \nabla_\Ho$ and which has the \textit{largest} Weitzenb\"ock type curvature for the metric $g_\varepsilon$. Let us note that it is however more convenient to work with $\square_{\varepsilon }$ instead of $\square_{\varepsilon /2}$ because of the previous point; this does not change anyting in terms of curvature quantities and optimal constants.
\end{itemize}

\

The following theorem is the main result of the section

\begin{theorem}\label{Bochner}
Let $f \in C^\infty(\M)$, we have
\[
dLf=\square_\varepsilon df,
\]
and for    $\eta =df $, the following Bochner's inequality holds
\begin{align*}
& \frac{1}{2} L \| \eta \|_{\varepsilon}^2 -\langle \square_\varepsilon \eta , \eta \rangle_{\varepsilon} \\
 = &  \| \nabla_{\mathcal{H}} \eta  -\mathfrak{T}^\varepsilon_{\mathcal{H}} \eta \|_{\varepsilon}^2 + \left\langle \mathfrak{Ric}_{\mathcal{H}} (\eta), \eta \right\rangle_\mathcal{H} -\left \langle \delta_\mathcal{H} T (\eta) , \eta \right\rangle_\mathcal{V} +\frac{1}{\varepsilon} \langle \mathbf{J}^2 (\eta) , \eta \rangle_\mathcal{H} \\
   \ge & \frac{1}{n}\left( \mathbf{Tr}_\mathcal{H}  \nabla_\ch^\# \eta \right)^2 -\frac{1}{4} \mathbf{Tr}_\mathcal{H} (J^2_{\eta})+ \left\langle \mathfrak{Ric}_{\mathcal{H}} (\eta), \eta \right\rangle_\mathcal{H} -\left \langle \delta_\mathcal{H} T (\eta) , \eta \right\rangle_\mathcal{V} +\frac{1}{\varepsilon} \langle \mathbf{J}^2 (\eta) , \eta \rangle_\mathcal{H}
\end{align*}
\end{theorem}

\

The remainder of the section is devoted to the proof of this result.

\

Since the result is local, we can assume that the foliation comes from a totally geodesic submersion and work in the frame $\{X_1,\dots,X_n, Z_1,\dots,Z_m\}$  of Lemma \ref{frame}.
The dual coframe of $\{X_1,\dots,X_n, Z_1,\dots,Z_m\}$ will be denoted $\{ \theta_1,\dots, \theta_n, \nu_1,\dots,\nu_m \}$ and a generic  one-form $\eta$ will be written, $\eta=\sum_{i=1}^n f_i \theta_i+\sum_{\ee=1}^m g_\ee \nu_\ee$.

\

\begin{lemma}\label{local-compute}
At $x$,
\begin{itemize}
\item $\mathbf{Ricci} (X_i,X_k)= \sum_{j=1}^n\left( \frac{1}{2} X_j (\omega_{ik}^j -\omega_{ij}^k-\omega_{kj}^i  ) -X_i\omega_{jk}^j \right)$.
\item $\mathbf{Ricci} (Z_\ee,X_k)= -\sum_{j=1}^n Z_\ee \omega_{jk}^j= 0 $.            
\item $\J(\eta)=\sum_{i,j=1}^n\sum_{\ee=1}^m\gamma_{ji}^\ee(X_ig_\ee)\theta_j$.
\item $\delta_\ch T(\eta)=-\sum_{i,j=1}^n\sum_{\ee=1}^m\left(X_i\gamma_{ij}^\ee \right)f_j\nu_\ee.$
\item  $\label{eq-delta-ch}
\delta^*_\ch T(\eta)=\sum_{i,j=1}^n\sum_{k=1}^m\left(X_j\gamma_{ij}^k\right)g_k\theta_i.$
\item $\mathfrak{T}^\varepsilon_{X_i}\eta=\sum_{j=1}^n\sum_{\ee=1}^m\left(\gamma_{ij}^\ee g_\ee \theta_j
-\frac{1}{\varepsilon}\gamma_{ij}^\ee f_j\nu_\ee\right)$
\item $\langle\mathbf{J}^2 (\eta) , \eta \rangle_\mathcal{H}=- \sum_{\ee=1}^m \sum_{i=1}^n\left(\sum_{j=1}^n\gamma_{ij}^\ee f_j\right)^2$
\end{itemize}
\end{lemma}

\begin{proof}
The computations are routine. We just point out that the vanishing of $\mathbf{Ricci} (Z_\ee,X_k)$ comes from the fact that since $X_k$ and  $\nabla_{X_i}X_k$ are basic and $[X_i,Z_\ee]$ is tangent to the leaves, we have at $x$, $\nabla_{Z_\ee}X_k=\nabla_{Z_\ee}\nabla_{X_i}X_k=\nabla_{[X_i,Z_\ee]}X_k=0$.

\end{proof}

\begin{lemma}\label{lemma-sq-infty}
Let  $\square_\infty$ be the operator defined  on one-forms by
\[
\square_\infty=L+2\J-\Ric_\ch+\delta^*_\mathcal{H} T ,
\]
then for any $f\in C^\infty(\M)$,
\[
 dLf=\square_\infty df.
\]
\end{lemma}

\begin{proof}
We compute,
\begin{align}
\begin{aligned}  \label{eq-commutator}
 dLf - Ldf & = \sum_{i=1}^n ([X_i,L]f) \theta_i -(X_i f) L\theta_i - \sum_{j=1}^n 2(X_j X_i f) \nabla_{X_j} \theta_i   \\
   & +\sum_{\ee=1}^m  ([Z_\ee,L]f)\nu_\ee -(Z_\ee f)L\nu_\ee - \sum_{j=1}^n 2(X_j Z_\ee f) \nabla_{X_j} \nu_\ee .
\end{aligned}
\end{align}

Now, at the center $x$ of the frame, we have $ \nabla_{X_j} \theta_i =  \nabla_{X_j} \nu_\ee   =0$, and
\begin{align}
\begin{aligned} \label{eq-commutator2}
L \theta_i = \sum_{j,k=1}^n (-X_j\Gamma_{jk}^i ) \theta_k , \quad L \nu_\ee =\sum_{j=1}^n \sum_{k=1}^m (-X_j \beta_{jk}^\ee) \nu_k .
\end{aligned}
\end{align}
where the $\Gamma_{ij}^k$'s are  the Cristofell symbols of the Bott connection.
We also easily compute
\begin{align*}
 [Z_\ee ,L] f = \sum_{j=1}^m ( -\sum_{i=1}^n X_i \beta_{i\ee}^j  ) Z_j f
\end{align*}
and
\begin{align}
\begin{aligned} \label{eq-commutator3}
 [X_i,L] f =   \sum_{j=1}^n \sum_{\ee=1}^m  2 \gamma_{ij}^\ee X_j Z_\ee f + \sum_{\ee=1}^m (  \sum_{j=1}^n X_j\gamma_{ij}^\ee ) Z_\ee f  + \sum_{j,k=1}^n (X_k \omega_{ik}^j -X_i \omega_{jk}^k ) X_j f .
\end{aligned}
\end{align}

If we plug this  in (\ref{eq-commutator}), then the second line of (\ref{eq-commutator}) turns out to be $0$ and we deduce from Lemma \ref{local-compute} that at $x$, we have
\begin{align*}
 dLf - Ldf =  2\J(df) -\Ric_\ch(df) +\delta^*_\mathcal{H} T (df).
\end{align*}
This completes the proof.
\end{proof}

Let us consider the map $\mathcal{T} \colon \Gamma^\infty(\wedge^2 T^*\M)\to \Gamma^\infty( T^*\M)$ which is given in the local frame by,
\[
\mathcal{T}(\theta_i\wedge\theta_j)=-\gamma_{ij}^\ee\nu_\ee,\quad \mathcal{T}(\theta_i\wedge\nu_k)=\mathcal{T}(\nu_k\wedge\nu_\ee)=0.
\]

\begin{lemma}\label{lemma-sq-ep}
For $\varepsilon >0$
\[
\square_\varepsilon=\square_\infty-\frac{2}{\ep}\mathcal{T}\circ d.
\]
\end{lemma}

\begin{proof}
A direct computation shows that, at the center $x$,  for any $\eta=\sum_{j=1}^nf_j\theta_j+\sum_{k=1}^mg_k\nu_k$, we have
\begin{align*}
 & -(\nabla_\mathcal{H} -\mathfrak{T}_\mathcal{H}^\varepsilon)^* (\nabla_\mathcal{H} -\mathfrak{T}_\mathcal{H}^\varepsilon)\eta\\
=&
\sum_{i,j=1}^n\left(X_i^2f_j-\sum_{\ee=1}^mX_i(\gamma_{ij}^\ee g_\ee)
-  \sum_{k=1}^n(X_i \Gamma_{ij}^k)f_k 
\right)\theta_j \\
 &+
\sum_{\ee=1}^m\left(\sum_{i=1}^n X_i^2 g_\ee+\frac{1}{\ep}\sum_{i,j=1}^n X_i(\gamma_{ij}^\ee f_j)+\sum_{i=1}^n\sum_{k=1}^mX_i(\beta_{ik}^\ee g_k)\right)\nu_\ee
\\
& + \frac{1}{\ep}\sum_{i,j=1}^n\sum_{k=1}^m\left(X_i f_j-\sum_{\ee=1}^m \gamma_{ij}^\ee g_\ee\right)\gamma_{ij}^k\nu_k
 -
\sum_{i,h=1}^n\sum_{\ee=1}^m \left(X_i g_\ee+\frac{1}{\ep}\sum_{j=1}^n\gamma_{ij}^\ee f_j \right)\gamma_{ih}^\ee\theta_h
\end{align*}
Therefore, we have 
\[
\left( -(\nabla_\mathcal{H} -\mathfrak{T}_\mathcal{H}^\varepsilon)^* (\nabla_\mathcal{H} -\mathfrak{T}_\mathcal{H}^\varepsilon)-\frac{1}{\ep}\mathbf{J}^2\right)\eta
=
\left(L+2\J-\frac{2}{\ep}\mathcal{T}\circ d+\delta^*_\ch T-\frac{1}{\ep}\delta_\ch T\right)\eta.
\]
By using the definition of $\square_\infty$ we immediately obtain the conclusion.
\end{proof}

\begin{lemma}\label{thm-cd-sq}
For any   $\eta \in \Gamma^\infty(T^* \M)$,
\[
\frac{1}{2} L \| \eta \|_{\varepsilon}^2 -\langle \square_\varepsilon \eta , \eta \rangle_{\varepsilon} =  \| \nabla_{\mathcal{H}} \eta  -\mathfrak{T}^\varepsilon_{\mathcal{H}} \eta \|_{\varepsilon}^2 + \left\langle \mathfrak{Ric}_{\mathcal{H}} (\eta), \eta \right\rangle_\mathcal{H} -\left \langle \delta_\mathcal{H} T (\eta) , \eta \right\rangle_\mathcal{V} +\frac{1}{\varepsilon} \langle \mathbf{J}^2 (\eta) , \eta \rangle_\mathcal{H}.
\]
\end{lemma}

\begin{proof}
First note that
\[
\square_\varepsilon=\square_\infty-\frac{2}{\ep}\mathcal{T}\circ d=L+2\J-\Ric_\ch+\delta^*_\ch T-\frac{2}{\ep}\mathcal{T}\circ d,
\]
and since $\delta^*_\ch T(\eta)$ and $\Ric(\eta)$ are horizontal while $\delta_\ch T(\eta)$ is vertical, we have 
\[
\langle \mathfrak{Ric}_{\mathcal{H}} (\eta), \eta \rangle_\ep=\langle \mathfrak{Ric}_{\mathcal{H}} (\eta), \eta \rangle_\mathcal{H},
\quad
\langle \delta_\mathcal{H}^* T (\eta) , \eta \rangle_{\varepsilon}= \langle \delta_\mathcal{H} T (\eta) , \eta \rangle_\mathcal{V}.
\]
Therefore it is equivalent to show that
\begin{align}\label{eq-cd-eta}
\frac{1}{2} L \| \eta \|_{\varepsilon}^2 -\left\langle \left(L+2\J-\frac{2}{\ep}\mathcal{T}\circ d\right)\eta,\eta\right\rangle_\varepsilon
= \| \nabla_{\mathcal{H}} \eta  -\mathfrak{T}^\varepsilon_{\mathcal{H}} \eta \|_{\varepsilon}^2 +\frac{1}{ \varepsilon}\left\langle  \mathbf{J}^2 (\eta), \eta \right\rangle_{\mathcal{H}}.
\end{align}

We now compute that, at the center $x$,
\begin{align}\label{eq-cd-eta-1}
& \frac{1}{2} L \| \eta \|_{\varepsilon}^2 -\left\langle \left(L+2\J-\frac{2}{\ep} \mathcal{T}\circ d\right)\eta,\eta\right\rangle_\varepsilon \\ \notag
 =&
\|\nabla_\ch \eta\|^2_\ch+\ep\|\nabla_\ch \eta\|_\V^2+2\sum_{i,j=1}^n\sum_{\ee=1}^m\gamma_{ij}^\ee (X_ig_\ee)f_j
-2\sum_{i,j=1}^n\sum_{k,\ee=1}^m\gamma_{ij}^\ee(X_if_j-\frac{1}{2}\gamma_{ij}^kg_k)g_\ee.
\end{align}
and
\begin{align*}
  \| \nabla_{\mathcal{H}} \eta  -\mathfrak{T}^\varepsilon_{\mathcal{H}} \eta \|_{\varepsilon}^2
 & =\sum_{i,j=1}^n\left(X_if_j-\sum_{\ee=1}^m\gamma_{ij}^\ee g_\ee \right)^2
+\ep\sum_{\ee=1}^m \sum_{i=1}^n\left(X_ig_\ee+\frac{1}{\varepsilon}\sum_{j=1}^n\gamma_{ij}^\ee f_j\right)^2
\\
& = \|\nabla_\ch \eta\|^2_\ch+\ep\|\nabla_\ch \eta\|_\V^2-2\sum_{i,j=1}^n\sum_{\ee=1}^m(X_if_j)\gamma_{ij}^\ee g_\ee+\sum_{i,j=1}^n\sum_{k,\ee=1}^m\gamma_{ij}^\ee\gamma_{ij}^kg_\ee g_k
\\
& +2\ep\sum_{i=1}^n\sum_{\ee=1}^m(X_ig_\ee)\left(\frac{1}{\ep}\sum_{j=1}^n\gamma_{ij}^\ee f_j \right)
+\ep\sum_{\ee=1}^m \sum_{i=1}^n\left(\frac{1}{\varepsilon}\sum_{j=1}^n\gamma_{ij}^\ee f_j\right)^2
\end{align*}
The claim easily follows.
\end{proof}

\begin{proposition}\label{prop-ineq}
Let $f \in C^\infty(\M)$. With $\eta=df$, we have
\begin{align*}
 & \frac{1}{2} L \| \eta \|_{\varepsilon}^2 -\langle \square_\varepsilon \eta , \eta \rangle_{\varepsilon}  \\
   \ge & \frac{1}{n}\left( \mathbf{Tr}_\mathcal{H}  \nabla_\ch^\# \eta \right)^2 -\frac{1}{4} \mathbf{Tr}_\mathcal{H} (J^2_{\eta})+ \left\langle \mathfrak{Ric}_{\mathcal{H}} (\eta), \eta \right\rangle_\mathcal{H} -\left \langle \delta_\mathcal{H} T (\eta) , \eta \right\rangle_\mathcal{V} +\frac{1}{\varepsilon} \langle \mathbf{J}^2 (\eta) , \eta \rangle_\mathcal{H}
\end{align*}\end{proposition}
\begin{proof}
We have
\begin{align*}
\|\nabla_\ch\eta-\mathfrak{T}^\ep_\ch\eta\|_\ch^2=\sum_{i,j=1}^n \left( X_if_j-\sum_{\ee=1}^m\gamma_{ij}^\ee g_\ee\right)^2.
\end{align*}
We now compute
\begin{align*}
\sum_{i,j=1}^n \left( X_if_j-\sum_{\ee=1}^m\gamma_{ij}^\ee g_\ee\right)^2 &= \frac{1}{2} \sum_{i,j=1}^n \left( X_if_j-\sum_{\ee=1}^m\gamma_{ij}^\ee g_\ee\right)^2+\left( X_j f_i+\sum_{\ee=1}^m\gamma_{ij}^\ee g_\ee\right)^2 \\
&= \frac{1}{4} \sum_{i,j=1}^n \left( X_if_j+X_jf_i \right)^2+\left( X_j f_i-X_if_j+2 \sum_{\ee=1}^m\gamma_{ij}^\ee g_\ee\right)^2
\end{align*}
Since $\eta=df$, we have  at the center of the frame $x$,
\[
X_j f_i-X_if_j=-\sum_{\ee=1}^m\gamma_{ij}^\ee g_\ee.
\]
Thus,
 \begin{align*}
\|\nabla_\ch\eta-\mathfrak{T}^\ep_\ch\eta\|_\ch^2
=\|\nabla^\#_\ch\eta\|^2_\ch+\frac{1}{4}\sum_{i,j=1}^n\left( \sum_{\ee=1}^m\gamma_{ij}^\ee g_\ee\right)^2,
 \end{align*}
and from Cauchy-Schwarz inequality, we conclude
 \begin{align*}
   \| \nabla_{\mathcal{H}} \eta  -\mathfrak{T}^\varepsilon_{\mathcal{H}} \eta \|_{\varepsilon}^2
  & = \| \nabla^\#_{\mathcal{H}} \eta  \|_{\varepsilon}^2
  +\frac{1}{4}\sum_{i,j=1}^n\left( \sum_{\ee=1}^m\gamma_{ij}^\ee g_\ee\right)^2
  +\ep \| \nabla_{\mathcal{H}} \eta  -\mathfrak{T}^\varepsilon_{\mathcal{H}} \eta \|_\V^2 \\
 & \ge
  \frac{1}{n}\left( \mathbf{Tr}_\mathcal{H}  \nabla_\ch^\# \eta \right)^2 -\frac{1}{4} \mathbf{Tr}_\mathcal{H} (J^2_{\eta}).
 \end{align*}
\end{proof}

We are now finally in position to conclude the proof of Theorem \ref{Bochner}.

\

\textbf{Proof of Theorem \ref{Bochner}}

From Lemma \ref{lemma-sq-infty} and \ref{lemma-sq-ep} we immediately obtain that
\[
dLf=\square_\varepsilon df,
\]
for all $f \in C^\infty(\M)$. The second part is Lemma \ref{thm-cd-sq} and Proposition \ref{prop-ineq}.

\section{Stochastic completeness }

Throughout the section we consider, as above, a smooth connected manifold $\M$ which is equipped with a Riemannian foliation with bundle like metric $g$ and totally geodesic leaves. We also assume  that for every horizontal one-form $\eta \in \Gamma^\infty(\mathcal{H}^*)$,
\[
 \langle \mathfrak{Ric}_{\mathcal{H}} (\eta) , \eta  \rangle_\mathcal{H} \ge -K \| \eta \|^2_\mathcal{H} , \quad -\langle \mathbf{J}^2 \eta, \eta  \rangle_\mathcal{H} \le \kappa  \| \eta \|^2_\mathcal{H},
\]
with $K,\kappa \ge 0$.

\

We moreover assume that the metric $g$ is complete and that the horizontal distribution $\mathcal{H}$ of the foliation is bracket-generating and Yang-Mills (see Besse \cite{Besse}, Definition 9.35). The hypothesis that $\mathcal{H}$ is bracket generating implies that the horizontal Laplacian $L$ is hypoelliptic and it is easily seen that with our notations the Yang-Mills condition is equivalent to the fact that
\[
\delta_\mathcal{H} T=0.
\]
The operator
\[
\square_\varepsilon=-(\nabla_\mathcal{H} -\mathfrak{T}_\mathcal{H}^\varepsilon)^* (\nabla_\mathcal{H} -\mathfrak{T}_\mathcal{H}^\varepsilon)-\frac{1}{ \varepsilon} \mathbf{ J}^2 - \mathfrak{Ric}_{\mathcal{H}}.
\]
that we introduced in the previous section is then symmetric for the $(L^2,g_\varepsilon)$ product on sections. We observe that without the Yang-Mills condition, the operator $\square_\varepsilon$ is not symmetric and thus we can not use the classical theory of self-adjoint operators to define the semigroup generated by it. Lemma \ref{commu2} below would then need to be interpreted in a  completely different way. Even if we assume a lower bound on $\delta_\mathcal{H} T$, it is not clear how to prove the stochastic completeness (Theorem 4.2 below). Interestingly in \cite{BG}, though the method is completely different, the Yang-Mills is also crucially used to prove the stochastic completeness.

\

The completeness of the metric $g$ implies that the horizontal Laplacian $L$ is essentially self-adjoint on the space of smooth and compactly supported functions and that the operator $\square_\varepsilon$ is essentially self-adjoint on the space of smooth and compactly supported one-forms (see the argument in Lemma 4.3 of \cite{B}).

Since $ \square_\varepsilon$ is essentially self-adjoint, it admits a unique self-adjoint extension which generates thanks to the spectral theorem a semigroup $Q^\varepsilon_t=e^{ t \square_\varepsilon}$.   We will denote by $P_t=e^{ t L}$ the semigroup generated by $ L$. We have the following commutation property:

\begin{lemma}\label{commu2}
If $f \in C_0^\infty(\M)$, then for every $t \ge 0$,
\[
d P_t f=Q^\varepsilon_t df.
\]
\end{lemma}

\begin{proof}
Let $\eta_t =Q^\varepsilon_t df$. By essential self-adjointness, it is the unique solution in $(L^2,g_\varepsilon)$ of the heat equation
\[
\frac{\partial \eta}{\partial t}=  \square_\varepsilon \eta,
\]
with initial condition $\eta_0 =df$. From the fact that
\[
dL=\square_\varepsilon d,
\]
we see that $\alpha_t=dP_t f$ solves the heat equation
\[
\frac{\partial \alpha}{\partial t}=  \square_\varepsilon \alpha
\]
with the same initial condition $\alpha_0=df$. In order to conclude, we thus just need to prove that for every $t \ge 0$, $dP_tf$ is in $(L^2,g)$ (and thus also in $(L^2,g_\varepsilon)$). Let us denote by $L^V$ the vertical (leaf) Laplacian. The Laplace-Beltrami operator of $\M$ is therefore $\Delta=L+L^V$. Since the leaves are totally geodesic, $\Delta$ commutes with $L$ on $C^2$ functions (see \cite{BeBo}). Moreover from the spectral theorem, $L e^{t \Delta}$ maps $C_0^\infty(\M)$ into $L^2(\M)$. We deduce by essential self-adjointness that $ L e^{t \Delta}=e^{t \Delta} L$. Similarly we obtain $ e^{sL} e^{t \Delta}=e^{t \Delta} e^{sL}$ which implies $\Delta e^{sL}=e^{sL} \Delta$. This equality implies that $P_s f=e^{sL}f$ is in the domain of $\Delta$ and  
\[
 \int_\M (\Delta P_s f) P_s f d\mu =  \int_\M(  P_s \Delta f) P_s f d\mu 
\]

As a consequence we have that for every $t \ge 0$, $dP_tf$ is in $L^2$, and moreover 
\[
\int_\M \| dP_tf \|^2 d\mu \le \frac{1}{2} \left( \int_\M  \| \Delta f \|^2 d\mu +  \int_\M f^2 d\mu\right)
\]
\end{proof}

\begin{theorem}\label{complete}
For every $\varepsilon >0$, $ t\ge 0, x\in \M $ and $ f \in C_0^\infty(\M)$,
\[
\| dP_t f (x) \|_\varepsilon \le e^{\left( K +\frac{\kappa}{\varepsilon} \right) t} P_t \| d f \|_\varepsilon (x).
\]
As a consequence, the heat semigroup is conservative that is for every $t \ge 0$, $P_t 1=1$.
\end{theorem}

\begin{proof}
The idea is to use the Feynman-Kac  stochastic representation of $Q_t^{\varepsilon}$.

We denote by $(X_t)_{t\geq 0}$  the symmetric diffusion process generated by $\frac{1}{2}L$ denote by $\mathbf{e}$ its lifetime. In the sequel of the proof, the symbol $\circ d $ denotes the Stratonovitch stochastic derivative.

Consider the process $\tau_t^{\varepsilon}:T_{X_t}^*\mathbb{M}\rightarrow T^*_{X_0}\mathbb{M}$ which is the solution of the following covariant Stratonovitch stochastic differential equation:
\begin{equation}\label{tau_t}
\circ d[\tau_t^{\varepsilon}\alpha(X_t)]=\tau_t^{\varepsilon}\left( \nabla_{\circ dX_t}-\mathfrak{T}_{\circ dX_t}^{\varepsilon}-\frac{1}{2}
\left(\frac{1}{\varepsilon}\mathbf{J}^2 +\mathfrak{Ric}_{\mathcal{H}}\right)dt\right) \alpha(X_t),~~\tau_0^{\varepsilon}=\mathbf{Id},
\end{equation}
where $\alpha$ is any smooth one-form. It is seen from the Stratonovitch integration by parts formula that we have 
\begin{equation}\label{tau=M Theta}
\tau^{\varepsilon}_t=\mathcal{M}_{t}^{\varepsilon}\Theta_{t}^{\varepsilon}
\end{equation}
where the process $\Theta_t^{\varepsilon}: T_{X_t}^{*}\mathbb{M}\rightarrow T_{X_0}^{*}\mathbb{M}$ is the solution of the following covariant Stratonovitch stochastic differential equation: 
\begin{equation}\label{Theta equation}
\circ d[\Theta_t^{\varepsilon}\alpha(X_t)]= \Theta_t^{\varepsilon}(\nabla_{\circ dX_t}-\mathfrak{T}^{\varepsilon}_{\circ dX_t})\alpha(X_t),~~\Theta_0^{\varepsilon}=\mathbf{Id}
\end{equation}
 The multiplicative functional $(\mathcal{M}_t^{\varepsilon})_{t\geq 0}$ is the solution of the following ordinary differential equation
\begin{equation}\label{multiplicative function M_t}
\frac{d\mathcal{M}_t^{\varepsilon}}{dt}=-\frac{1}{2}\mathcal{M}_t^{\varepsilon}\Theta_t^{\varepsilon}\left(\frac{1}{\varepsilon}\mathbf{J}^2+\mathfrak{Ric}_{\mathcal{H}} \right)(\Theta_t^{\varepsilon})^{-1}, ~~\mathcal{M}_0^{\varepsilon}=\mathbf{Id}.
\end{equation} 

Since $\mathfrak{T}^{\varepsilon}$ is skew-symmetric,  $\nabla - \mathfrak{T}^{\varepsilon} $ is a $g_\varepsilon$ metric connection and   $\Theta_t^{\varepsilon}$ is thus  an isometry for the Riemannian metric $g_{\varepsilon}$.  From Gronwall's lemma, we deduce then that the following pointwise bound holds
\[
\| \tau^\varepsilon_t \alpha (X_t) \|_{ \varepsilon} \le e^{\frac{1}{2}\left( K+\frac{\kappa}{ \varepsilon} \right)t} \| \alpha (X_t) \|_{ \varepsilon}.
\]
By the Feynman-Kac formula, we have for every  smooth and compactly supported one-form
\[
Q_{t/2}^\varepsilon \eta (x)=\mathbb{E}_x \left( \tau_t^\varepsilon \eta (X_t) \mathbf{1}_{t < \mathbf{e}} \right),
\]
where $\mathbb{E}_x$ denotes the expectation conditionally to $X_0=x$.
Since $dP_t=Q_t^\varepsilon d$, it follows easily that
\[
\| dP_t f (x) \|_\varepsilon \le e^{\left( K +\frac{\kappa}{\varepsilon} \right) t} P_t \| d f \|_\varepsilon (x).
\]
It is well-known that this type of gradient bound implies the stochastic completeness of $P_t$ (see \cite{Bak}).
\end{proof}

\section{Curvature-dimension inequality, Li-Yau estimates and Bonnet-Myers type theorem}

Let $\M$ be a smooth, connected  manifold with dimension $n+m$. We assume that $\bM$ is equipped with a Riemannian foliation $\mathcal{F}$ with bundle like metric $g$ and totally geodesic  $m$-dimensional leaves for which the horizontal distribution is Yang-Mills. We also assume that $\M$ is complete and that globally on $\M$, for every $\eta_1 \in \Gamma^\infty(\mathcal{H}^*)$ and $\eta_2 \in \Gamma^\infty(\mathcal{V}^*)$,
\[
\langle \mathfrak{Ric}_{\mathcal{H}} (\eta_1), \eta_1 \rangle_{\mathcal{H}} \ge  \rho_1 \| \eta_1 \|^2_{\mathcal{H}}, \quad -\langle \mathbf{J}^2 \eta_1, \eta_1  \rangle_\mathcal{H} \le \kappa  \| \eta_1 \|^2_\mathcal{H},\quad -\frac{1}{4} \mathbf{Tr}_\mathcal{H} (J^2_{\eta_2})\ge \rho_2 \| \eta_2 \|^2_\mathcal{V},
\]
for some $\rho_1 \in \mathbb{R}$, $\kappa,\rho_2 >0$. The third assumption can be thought as a uniform bracket generating condition of the horizontal distribution $\mathcal{H}$ and from H\"ormander's theorem, it implies that the horizontal Laplacian $L$ is a subelliptic diffusion operator. We insist that for the following results below to be true, the positivity of $\rho_2$ is required. 

\

We introduce the following operators defined for $f,g \in C^\infty(\M)$,
\[
\Gamma(f,g)=\frac{1}{2} ( L(fg) -gLf-fLg)=\langle \nabla_\mathcal{H} f , \nabla_\mathcal{H} g\rangle_\mathcal{H}
\]

\[
\Gamma^\mathcal{V} (f,g)=\langle \nabla_\mathcal{V} f , \nabla_\mathcal{V} g\rangle_\mathcal{V}
\]
and their iterations which are defined by
\[
\Gamma_2(f,g)=\frac{1}{2} ( L(\Gamma(f,g)) -\Gamma(g,Lf)-\Gamma(f,Lg))
\]
\[
\Gamma^\mathcal{V}_2(f,g)=\frac{1}{2} ( L(\Gamma^\mathcal{V}(f,g)) -\Gamma^\mathcal{V}(g,Lf)-\Gamma^\mathcal{V}(f,Lg))
\]
As a consequence of Theorem \ref{Bochner}, we obtain
\begin{theorem}\label{CD}
For every $f,g \in C^\infty(\M)$, and $\varepsilon >0$,
\begin{align}\label{CDI}
\Gamma_2(f,f)+\varepsilon \Gamma^\mathcal{V}_2(f,f)\ge \frac{1}{n} (Lf)^2 +\left( \rho_1 -\frac{\kappa}{\varepsilon}\right) \Gamma(f,f)+\rho_2 \Gamma^\mathcal{V} (f,f),
\end{align}
and
\[
\Gamma (f, \Gamma^\mathcal{V} (f))=\Gamma^\mathcal{V}  (f, \Gamma (f)).
\]
\end{theorem}
\begin{proof}
From Theorem 3.1, we have for every   $\eta =df$,
\begin{align*}
 \frac{1}{2} L \| \eta \|_{\varepsilon}^2 -\langle \square_\varepsilon \eta , \eta \rangle_{\varepsilon}    \ge & \frac{1}{n}\left( \mathbf{Tr}_\mathcal{H}  \nabla_\ch^\# \eta \right)^2 -\frac{1}{4} \mathbf{Tr}_\mathcal{H} (J^2_{\eta})+ \left\langle \mathfrak{Ric}_{\mathcal{H}} (\eta), \eta \right\rangle_\mathcal{H}  +\frac{1}{\varepsilon} \langle \mathbf{J}^2 (\eta) , \eta \rangle_\mathcal{H}.
\end{align*}
Using this inequality and taking into account the assumptions
\[
\langle \mathfrak{Ric}_{\mathcal{H}} (\eta_1), \eta_1 \rangle_{\mathcal{H}} \ge  \rho_1 \| \eta_1 \|^2_{\mathcal{H}}, \quad -\langle \mathbf{J}^2 \eta_1, \eta_1  \rangle_\mathcal{H} \le \kappa  \| \eta_1 \|^2_\mathcal{H},\quad -\frac{1}{4} \mathbf{Tr}_\mathcal{H} (J^2_{\eta_2})\ge \rho_2 \| \eta_2 \|^2_\mathcal{V},
\]
immediately yields the expected result. The intertwining $\Gamma (f, \Gamma^\mathcal{V} (f))=\Gamma^\mathcal{V}  (f, \Gamma (f))$ is proved in Appendix A in \cite{Elworthy} and easy to check in a local frame (see also \cite{IHP}).
\end{proof}

Let us  recall that the generalized curvature dimension inequality CD$(\rho_1,\rho_2,\kappa,d)$ introduced in \cite{BG} exactly writes as  \eqref{CDI}. Thus, combining  Theorems \ref{complete} and \ref{CD}, we see then that all the results proved  in the works \cite{BB,BBG,BBGM, BG, BG2, BK, Kim} apply. We obtain therefore, among many other things, the following results which are completely new in the context of Riemannian foliations:

\begin{itemize}
\item[1)] \textbf{Li-Yau type inequalities:} (\cite{BG}) For any bounded $f \in C^\infty(\bM)$, such that $f, \sqrt{\Gamma(f)}$, $\sqrt{\Gamma^\mathcal{V}(f)} \in L^2_\mu(\bM)$, $f  \ge 0$, $f \neq 0$, the following inequality holds for $t>0$:
\begin{align*}
 & \Gamma (\ln P_t f) +\frac{2 \rho_2}{3}  t \Gamma^\mathcal{V} (\ln P_t f)\\
  \le & \
\left(1+\frac{3\kappa}{2\rho_2}-\frac{2\rho_1}{3} t\right)
\frac{LP_t f}{P_t f} +\frac{n\rho_1^2}{6} t
-\frac{n \rho_1}{2}\left(
1+\frac{3\kappa}{2\rho_2}\right) +\frac{n\left(
1+\frac{3\kappa}{2\rho_2}\right)^2}{2t}.
\end{align*}
\item[2)] \textbf{Gaussian lower and  upper bounds for the horizontal heat kernel:}  (\cite{BBG}) If  $\rho_1 \ge 0$, then for any $0<\ve <1$
there exists a constant $C(\ve) = C(n,\kappa,\rho_2,\ve)>0$, which tends
to $\infty$ as $\ve \to 0^+$, such that for every $x,y\in \bM$
and $t>0$ one has
\[
\frac{C(\ve)^{-1}}{\mu(B(x,\sqrt
t))} \exp
\left(-\frac{D d(x,y)^2}{n(4-\ve)t}\right)\le p(x,y,t)\le \frac{C(\ve)}{\mu(B(x,\sqrt
t))} \exp
\left(-\frac{d(x,y)^2}{(4+\ve)t}\right).
\]
Here $D= \left(
1+\frac{3\kappa}{2\rho_2}\right)n$ and $d(x,y)$ is the sub-Riemannian distance between $x$ and $y$.
\item[3)]  \textbf{Bonnet-Myers  theorem:} (\cite{BG})  Suppose that $\rho_1 > 0$.
Then, the manifold $\mathbb{M}$ is compact   and the sub-Riemannian diameter of $\M$ satisfies the bound
\begin{align}\label{myers} \text{diam}\ \bM \le 2\sqrt{3} \pi \sqrt{
\frac{\kappa+\rho_2}{\rho_1\rho_2} \left(
1+\frac{3\kappa}{2\rho_2}\right)n }.
\end{align}
\end{itemize}

We mention that in the Bonnet-Myers theorem, the bound 
\begin{align*} \text{diam}\ \bM \le 2\sqrt{3} \pi \sqrt{
\frac{\kappa+\rho_2}{\rho_1\rho_2} \left(
1+\frac{3\kappa}{2\rho_2}\right)n }.
\end{align*}
is not sharp, as can be checked in some examples like the Hopf fibrations (see \cite{BW1,BW2} for the computation of the diameters of the Hopf fibrations). This is because the method we use in \cite{BG} is an adaption of the energy-entropy inequality methods developed by Bakry in \cite{bakry-stflour}. Even in the Riemannian case, Bakry methods are known to lead to non sharp constants.

\

Finally, at last, we observe that the methods of \cite{BK2} can be adapted to the present framework and that the following result can be proved:

\begin{proposition}
Assume $\rho_1>0$. Then the first eigenvalue $\lambda_1$ of the horizontal Laplacian $-L$  satisfies
\begin{align*}
 \lambda_1 \geq  \frac{\rho_1}{1-\frac{1}{n}+\frac{3\kappa}{4\rho_2}}.
\end{align*}

\end{proposition}

As a consequence of the Obata theorem proved in \cite{BK2}, this bound is sharp.


\begin{thebibliography}{10}
 
 \bibitem{agrachev}
A. Agrachev, P. Lee,  \emph{Generalized Ricci curvature bounds for three dimensional contact subriemannian manifolds}, To appear Math. Ann., (2014)

\bibitem{agrachev2} 
A. Agrachev, P. Lee,  \emph{Bishop and Laplacian comparison theorems on three dimensional contact sub-riemannian manifolds with symmetry},  Journal of Geometric Analysis, July 2013. 

\bibitem{Bak}
D.  Bakry,
\emph{Un crit\`ere de non-explosion pour certaines diffusions sur une vari\'et\'e riemannienne compl\`ete}. (French. English summary) [A non-explosion criterium for some diffusions on a complete Riemannian manifold]
C. R. Acad. Sci. Paris S\'er. I Math. 303 (1986), no. 1, 23-26.

\bibitem{bakry-stflour}
D. Bakry, \emph{L'hypercontractivit\'e et son utilisation en
th\'eorie des semigroupes}, Ecole d'Et\'e de Probabilites de
St-Flour, Lecture Notes in Math, (1994).

\bibitem{BR}
D. Barilari \& L. Rizzi, \emph{A formula  for Popp's volume in sub-Riemannian geometry}, 2013, Analysis and Geometry in Metric Spaces, 1, pp. 42-57

\bibitem{IHP}
F. Baudoin, \emph{Sub-Laplacians and hypoelliptic operators on totally geodesic Riemannian foliations}, 2014, Insitute Henri Poincare course, to appear in EMS monographs

 \bibitem{B}
 F. Baudoin, \emph{Stochastic analysis on sub-Riemannian manifolds with transverse symmetries}, 2014, To appear in Annals of Probability,  Arxiv preprint, http://arxiv.org/abs/1402.4490

\bibitem{BB}
F. Baudoin \&  M. Bonnefont, \emph{Log-Sobolev inequalities for subelliptic operators satisfying a generalized curvature dimension inequality}, Journal of Functional Analysis, Volume 262 ~(2012), 2646--2676.

\bibitem{BBG}
F. Baudoin,  M. Bonnefont \& N. Garofalo, \emph{A sub-Riemannian curvature-dimension inequality, volume doubling property and the Poincar\'e inequality}, Math. Ann. 358 (2014), no. 3-4, 833-860.

\bibitem{BBGM}
F. Baudoin,  M. Bonnefont, I. Munive \& N. Garofalo, \emph{Volume and distance comparison theorems for sub-Riemannian manifolds}, To appear in Journal of Functional Analysis (2014), Arxiv preprint, http://arxiv.org/abs/1211.0221

 \bibitem{BG}
F. Baudoin \& N. Garofalo, \emph{Curvature-dimension inequalities and Ricci lower bounds for sub-Riemannian manifolds with transverse symmetries}, To appear in JEMS, Arxiv preprint, http://arxiv.org/abs/1101.3590

\bibitem{BG2}
F. Baudoin \& N. Garofalo,
\emph{A note on the boundedness of Riesz transform for some subelliptic operators.} Int. Math. Res. Not. IMRN 2013, no. 2, 398Ð421.

\bibitem{BK}
F. Baudoin \& B. Kim, \emph{Sobolev, Poincar\'e and isoperimetric inequalities for subelliptic diffusion operators satisfying a generalized curvature dimension inequality}, Revista Matematica Iberoamericana, 30, (2014), 1, 109-131

\bibitem{BK2}
F. Baudoin \& B. Kim, \emph{The Lichnerowicz-Obata theorem on sub-Riemannian manifolds with transverse symmetries,} 2014, To appear in Journal of Geometric Analysis

\bibitem{BW1}
F. Baudoin \& J. Wang: \emph{The subelliptic heat kernel on the CR sphere}, Math. Z., 275, (2013), 1-2, 135-150

\bibitem{BW2}
F. Baudoin \& J. Wang: \emph{The subelliptic heat kernels of the quaternionic Hopf fibration}, Potential Analysis, 41, no 3, 959-982, (2014)

\bibitem{Be}
L. B\'erard-Bergery, \emph{Sur certaines fibrations d'espaces homog\`enes riemanniens}, Compositio Math, \textbf{30}, 43-61, (1975)


\bibitem{BeBo}
L. B\'erard-Bergery, J.P. Bourguignon,  \emph{Laplacians and Riemannian submersions with totally geodesic fibres}. Illinois J. Math. 26 (1982), no. 2, 181-200

\bibitem{Besse} A. Besse, \emph{Einstein manifolds.}  Reprint of the 1987 edition. Classics in Mathematics. Springer-Verlag, Berlin, 2008. xii+516 pp.

\bibitem{Elworthy} K.D. Elworthy, \emph{Decompositions of diffusion operators and related couplings}, preprint 2014

\bibitem{GL} M. Gordina, T. Laetsch, \emph{Sub-Laplacians on sub-Riemannian manifolds}, Arxiv 2014

\bibitem{Hladky} R. Hladky, \emph{Connections and Curvature in sub-Riemannian geometry}. Houston J. Math, 38 (2012), no. 4, 1107-1134

\bibitem{Kim} B. Kim, \emph{Poincar\'e inequality and the uniqueness of solutions for the heat equation associated with subelliptic diffusion operators}, http://arxiv.org/abs/1305.0508

\bibitem{Oneill} B. O'Neill, \emph{The fundamental equations of a submersion}, Michigan Math. J. 13 (1966) 459-469.

\bibitem{Tondeur} P.  Tondeur, \emph{Foliations on Riemannian manifolds}.  Universitext. Springer-Verlag, New York, 1988. xii+247 pp

\bibitem{Tondeur2} P.  Tondeur, \emph{Geometry of foliations}.  Monographs in Mathematics, 90. BirkhŠuser Verlag, Basel, 1997. viii+305 pp.

\end{thebibliography}
\end{document}